\documentclass[reqno,12pt,letterpaper]{amsart}
\usepackage{amsmath,amssymb,amsthm,graphicx,mathrsfs,url}
\usepackage[usenames,dvipsnames]{color}
\usepackage[colorlinks=true,linkcolor=Red,citecolor=Green]{hyperref}
\usepackage{amsxtra}

\def\?[#1]{\textbf{[#1]}\marginpar{\Large{\textbf{??}}}}

\let\epsilon=\varepsilon 

\newtheorem{thm}{Theorem}
\newtheorem{prop}{Proposition}
\newtheorem{defi}[prop]{Definition}

\newtheorem{lem}[prop]{Lemma}
\newtheorem{corr}[prop]{Corollary}

\numberwithin{equation}{section}

\DeclareMathOperator{\Spec}{Spec}

\DeclareMathOperator{\Op}{Op}

\DeclareMathOperator{\rank}{rank}

\DeclareMathOperator{\supp}{supp}

\title{Exact Control for Schr\"odinger Equation on Torus}
\author{Zhongkai Tao}
\email{tzk320581@berkeley.edu}
\address{Department of Mathematics, University of California, Berkeley, CA 94720, USA}
\address{Department of Mathematics and Statistics, Xi'an Jiaotong University, Xi'an, Shaanxi, China}

\begin{document}

\begin{abstract}
For standard torus $\mathbb{T}^2=\mathbb{R}^2/\mathbb{Z}^2$, we prove observability for free Schr\"odinger equation from a ball of radius $\epsilon$ with explicit dependence of the observability constant on $\epsilon$.
\end{abstract}

\maketitle

\section{Introduction}

We will follow some methods of Bourgain-Burq-Zworski \cite{BBZ}\cite{BuZw} and Jin \cite{Jin} to prove a quantitative version of observability result for the Schr\"odinger equation on the 2-dimensional standard torus.
\begin{thm}[Semiclassical Observability Estimate]\label{thm1}
Let $\mathbb{T}^2=\mathbb{R}^2/\mathbb{Z}^2$, and $\Omega_\epsilon=B(0,\epsilon)\subset\mathbb{T}^2$.
Then for any $\delta>0$, there exists numerical constant $C$ and $h_0=\epsilon^{16+\delta}$ such that for $0<h<h_0, \epsilon\ll 1$,
$$\|u\|_{L^2(\mathbb{T}^2)}\leq C\epsilon^{-4}\|u\|_{L^2(\Omega_\epsilon)}+C\epsilon^{-2}h^{-2}\|(-h^2\Delta-1)u\|_{L^2(\mathbb{T}^2)}.$$
\end{thm}
From Theorem \ref{thm1} we deduce the classical version
\begin{thm}\label{thm2}
On the torus $\mathbb{T}^2=\mathbb{R}^2/\mathbb{Z}^2$ we have
$$\|u_0\|_{L^2(\mathbb{T}^2)}^2\leq C_{\Omega_\epsilon}\int_0^{\frac{1}{2\pi}}\|e^{it\Delta}u_0\|_{L^2(\Omega_\epsilon)}^2dt$$
for $\Omega_\epsilon=B(0,\epsilon)$ and $C_{\Omega_\epsilon}=\exp\exp\frac{C\log{\epsilon^{-1}}}{\log\log{\epsilon^{-1}}}$ with some constant $C$ independent of $\epsilon$.
\end{thm}
\subsection{Historical Remark}
The control for Schr\"odinger equation is first shown by Lebeau \cite{Le} under the following geometric control condition
$$\mbox{There exists } T>0 \mbox{ such that every geodesic of length } T \mbox{ intersects } \Omega.$$
In general, the geometric control condition is not necessary. The observability estimate in the case of flat tori is shown by Jaffard \cite{Jaf} and Haraux \cite{Ha} in dimension two and by Komornik \cite{Ko} in higher dimensions. In dimension two, Burq-Zworski \cite{BuZw} extended the result to Schr\"odinger equation with smooth potential. Bourgain-Burq-Zworski \cite{BBZ} further extended it to the case of $L^2$ potential. In higher dimensions, the result is shown by Anantharaman-Maci\`{a} \cite{AnMa} with some class of potentials including continuous ones.

For compact negatively curved surfaces, the observability by any nonempty open set is proved by Dyatlov-Jin-Nonnenmacher \cite{DJN}.

All the above results do not provide an exact constant for torus. However, the observability estimate is proved for any $T>0$. We expect that our exact constant is valid for any $T>0$ but are not able to prove it for some technical reasons.

Theorem \ref{thm1} gives a lower bound on quantum limits on the standard torus. A better bound can be provided by the explicit description of the quantum limits by Jakobson \cite{Jak}.

\noindent\textbf{Acknowledgements.} This note is written based on an undergraduate research project supervised by Professor Semyon Dyatlov at Berkeley in 2019. The author would like to thank him for introducing this topic and a lot of helpful discussions. We would like to thank Nicolas Burq, Aleksandr Logunov, Ping Xi and Maciej Zworski for helpful discussions. The research was supported in part by the National Science Foundation CAREER grant DMS-1749858.
\section{Estimate in dimension one}
In this section we show the following observability estimate for the inhomogeneous Helmholtz equation.
\begin{prop}\label{prop1}
Let $\omega_\epsilon=(-\epsilon,\epsilon)\times [0,1]\subset \mathbb{T}^2=\mathbb{R}^2/\mathbb{Z}^2$, then for any $u\in H^2(\mathbb{T}^2)$ and $h>0$
$$\|u\|_{L^2(\mathbb{T}^2)}^2\leq C\epsilon^{-3}\|u\|_{L^2(\omega_\epsilon)}^2+4h^{-4}\|(-h^2\Delta-1)u\|_{L^2(\mathbb{T}^2)}^2.$$
\end{prop}
\begin{proof}
{\em Step 1}

We follow the method in \cite{BuZw2} to prove an estimate in dimension one. Let $\mathbb{T}^1=\mathbb{R}/\mathbb{Z}$, for any $v\in H^2(\mathbb{T}^1)$ and $z\in\mathbb{R}$, we claim
\begin{align}\label{2.1}
    \|v\|_{L^2(\mathbb{T}^1)}^2\leq C\epsilon^{-3}\|v\|_{L^2((-\epsilon,\epsilon))}^2+4h^{-4}\|(-h^2\partial_x^2-z)v\|_{L^2(\mathbb{T}^1)}^2.
\end{align}
Denote $f=(-h^2\partial_x^2-z)v$, we separate the proof into two cases:\\
Case 1: $z\leq 0$. In this case we have
\begin{align*}
    h^2\|\partial_x v\|_{L^2(\mathbb{T}^1)}^2&\leq\int_{\mathbb{T}^1}((-h^2\partial_x^2-z)\bar{v})vdx\\
    &\leq \|f\|_{L^2(\mathbb{T}^1)}\|v\|_{L^2(\mathbb{T}^1)}.
\end{align*}
Then
\begin{align*}
    |v(x)|&\leq \left|\int_t^x \partial_xv(y)dy\right|+|v(t)|\\
    &\leq \|\partial_x v\|_{L^2(\mathbb{T}^1)}+|v(t)|.
\end{align*}
So
\begin{align*}
    \|v\|_{L^2(\mathbb{T}^1)}^2&\leq 2\|\partial_xv\|_{L^2(\mathbb{T}^1)}^2+\epsilon^{-1}\|v\|_{L^2((-\epsilon,\epsilon))}^2\\
    &\leq \frac{2}{h^2}\|v\|_{L^2(\mathbb{T}^1)}\|f\|_{L^2(\mathbb{T}^1)}+\epsilon^{-1}\|v\|_{L^2((-\epsilon,\epsilon))}^2\\
    &\leq \frac{1}{2}\|v\|_{L^2(\mathbb{T}^1)}^2+\frac{2}{h^4}\|f\|_{L^2(\mathbb{T}^1)}^2+\epsilon^{-1}\|v\|_{L^2((-\epsilon,\epsilon))}^2.
\end{align*}
So
\begin{align*}
    \|v\|_{L^2(\mathbb{T}^1)}^2\leq \frac{4}{h^4}\|f\|_{L^2(\mathbb{T}^1)}^2+2\epsilon^{-1}\|v\|_{L^2((-\epsilon,\epsilon))}^2.
\end{align*}
\\Case 2: $z>0$\\
First choose $\chi\in C_0^\infty(\mathbb{T}^1)$ such that $\chi=0$ on $B(0,\frac{\epsilon}{3})$ and $\chi=1$ on $\mathbb{T}^1\setminus B(0,\frac{\epsilon}{2})$ with $|\chi^{(k)}(x)|\leq \frac{C_k}{\epsilon^k}, \forall k\in \mathbb{N}$.
We then have
$$(-h^2\partial_x^2-z)(\chi v)=h^2\partial_x^2\chi v-2h^2\partial_x(\partial_x\chi v)+\chi f=\tilde{f}.$$
The solution of the ODE is
$$\chi(x) v(x)=-\frac{1}{h\sqrt{z}}\int_0^x \sin{\frac{\sqrt{z}(x-t)}{h}}\tilde{f}(t)dt.$$
For each term we have
\begin{align*}
 \left|-\frac{1}{h\sqrt{z}}\int_0^x \sin{\frac{\sqrt{z}(x-t)}{h}}h^2\partial_x^2\chi v(t)dt\right|\leq   \|\partial_x^2\chi v\|_{L^1((-\epsilon,\epsilon))}\leq \frac{C}{\epsilon^\frac{3}{2}}\| v\|_{L^2((-\epsilon,\epsilon))},
\end{align*}
\begin{align*}
      \left|-\frac{1}{h\sqrt{z}}\int_0^x \sin{\frac{\sqrt{z}(x-t)}{h}}h^2\partial_t(\partial_x\chi v)(t)dt\right|&=\left|-\frac{1}{h\sqrt{z}}\int_0^x \partial_t\left(\sin{\frac{\sqrt{z}(x-t)}{h}}\right)h^2\partial_x\chi v(t)dt\right|\\
      &\leq \|\partial_x\chi v(t)\|_{L^1((-\epsilon,\epsilon))}\\
      &\leq \frac{C}{\epsilon^\frac{1}{2}}\|v\|_{L^2((-\epsilon,\epsilon))},
\end{align*}
\begin{align*}
    \left|-\frac{1}{h\sqrt{z}}\int_0^x \sin{\frac{\sqrt{z}(x-t)}{h}}\chi f(t)dt\right|\leq \frac{1}{h^2}\|f\|_{L^1(\mathbb{T}^1)}.
\end{align*}
Put them together we get
\begin{align*}
    \| v\|_{L^2(\mathbb{T}^1)}&\leq \|\chi v\|_{L^2(\mathbb{T}^1)}+\|(1-\chi )v\|_{L^2(\mathbb{T}^1)}\\
    &\leq \frac{C}{\epsilon^\frac{3}{2}}\| v\|_{L^2((-\epsilon,\epsilon))}+\frac{C}{\epsilon^\frac{1}{2}}\|v\|_{L^2((-\epsilon,\epsilon))} +\frac{1}{h^2}\|f\|_{L^1(\mathbb{T}^1)}+\|v\|_{L^2((-\epsilon,\epsilon))}\\
    &\leq \frac{C}{\epsilon^\frac{3}{2}}\| v\|_{L^2((-\epsilon,\epsilon))}+\frac{1}{h^2}\|f\|_{L^2(\mathbb{T}^1)}.
\end{align*}
\\
{\em Step 2}

Let $g=(-h^2\Delta-1)u$, we prove the 2-dimensional estimate by Fourier expansion in $y$.

Decompose $u=\sum\limits_{k\in\mathbb{Z}}u_k(x)e_k(y)$, and $g=\sum\limits_{k\in\mathbb{Z}}g_k(x)e_k(y)$ where $e_k(y)=e^{2k\pi i y}$, then
$$(-h^2\partial_x^2+(2k\pi)^2h^2-1)u_k=g_k.$$
The proof follows from the one-dimension estimate \eqref{2.1}
$$\|u_k\|_{L^2(\mathbb{T}^1)}^2\leq C\epsilon^{-3}\|u_k\|_{L^2(B(0,\epsilon))}^2+4h^{-4}\|g_k\|_{L^2(\mathbb{T}^1)}^2.$$
\end{proof}

\section{Semiclassical preliminaries}
In this section we recall some semiclasscial preliminaries we will use. The general reference is \cite{Zw}. Throughout this section, we take $\mathbb{T}^n=\mathbb{R}^n/(L\mathbb{Z})^n$ for some $L\geq 1$.
\subsection{$L^2$ boundedness of pseudo-differential operators}
We will recall several properties related to $L^2$ boundedness of pseudo-differential operators. First we recall the definition of Weyl quantization.
\begin{defi}
Let $a(x,\xi)\in S^m(T^*\mathbb{T}^n)$, the Weyl quantization is defined as
$$\Op^w_h(a)u(x)=\frac{1}{(2\pi h)^n}\int_{\mathbb{R}^n}\int_{\mathbb{R}^n}a\left(\frac{x+y}{2},\xi\right)e^{i(x-y)\xi/h}u(y)dyd\xi.$$
$\Op^w_h(a)$ is called an $m$-th order pseudo-differential operator.
\end{defi}
$0$-th order pseudodifferential operators are bounded on $L^2(\mathbb{T}^n)$. In fact, we have
\begin{lem}\label{lem3}
If $a\in S^0(T^*\mathbb{T}^n)$, then $\Op^w_h(a):L^2(\mathbb{T}^n)\to L^2(\mathbb{T}^n)$ is bounded with
$$\|\Op^w_h(a)\|\leq C\sum\limits_{|\alpha|\leq Kn}h^{\frac{|\alpha|}{2}}\|\partial^\alpha a\|_{L^\infty}$$
for some universal constant $K$.
\end{lem}
\begin{proof}
The proof follows from the proof of \cite[Theorem 4.23, Theorem 5.5]{Zw}.
\end{proof}

Since we will need to estimate $L^2$ bound for remainders in composition formula, we prove an estimate for the composition formula of pseudo-differential operators.
\begin{lem}
Let $A(D)=\frac{1}{2}\langle QD,D\rangle$ with $Q$ a real nonsingular symmetric matrix.
Suppose $a\in S^0(\mathbb{R}^n)$, then
$$\sum\limits_{|\alpha|\leq N}h^{\frac{|\alpha|}{2}}\|\partial^\alpha e^{ihA(D)}a\|_{L^\infty}\leq C\sum\limits_{|\alpha|\leq N+n+1}h^{\frac{|\alpha|}{2}}\|\partial^\alpha a\|_{L^\infty}.$$
\end{lem}
\begin{proof}
We just need to prove for $N=0$. Let $\chi\in C^\infty_0(\mathbb{R}^n)$ be a cutoff function near $0$ (i.e. $\chi(x)=1$ in a neighbourhood of $0$), then
\begin{align*}
    e^{ihA(D)}a
    &=\frac{C}{h^\frac{n}{2}}\int_{\mathbb{R}^n}e^{\frac{i\phi(w)}{h}}a(z-w)dw\\
    &=\frac{C}{h^\frac{n}{2}}\int_{\mathbb{R}^n}e^{\frac{i\phi(w)}{h}}\chi\left(\frac{w}{\sqrt{h}}\right)a(z-w)dw+\frac{C}{h^\frac{n}{2}}\int_{\mathbb{R}^n}e^{\frac{i\phi(w)}{h}}\left(1-\chi\left(\frac{w}{\sqrt{h}}\right)\right)a(z-w)dw\\
    &=A_1+A_2
\end{align*}
where $\phi(w)=-\frac{1}{2}\langle Q^{-1}w,w\rangle$.

We have $|A_1|\leq C\left|\int_{\mathbb{R}^n}e^{i\phi(w)}\chi(w)a(z-\sqrt{h}w)dw\right|\leq C\|a\|_{L^\infty}$. And let $L=\frac{\langle \partial \phi, D\rangle}{|\partial \phi|^2}$
\begin{align*}
    |A_2|&\leq C\left|\int_{\mathbb{R}^n}e^{i\phi(w)}(1-\chi(w))a(z-\sqrt{h}w)dw\right|\\
    &=C\left|\int_{\mathbb{R}^n}(L^{n+1}e^{i\phi(w)})(1-\chi(w))a(z-\sqrt{h}w)dw\right|\\
    &=C\left|\int_{\mathbb{R}^n}e^{i\phi(w)}(L^T)^{n+1}((1-\chi(w))a(z-\sqrt{h}w))dw\right|\\
    &\leq C\sum\limits_{|\alpha|\leq n+1}h^{\frac{|\alpha|}{2}}\|\partial^\alpha a\|_{L^\infty}.
\end{align*}

\end{proof}
Now in general, we have
\begin{align*}
    e^{ihA(D)}a&=\sum\limits_{k=0}^N\frac{i^kh^k}{k!}A(D)^ka+\frac{i^{N+1}h^{N+1}}{N!}\int_0^1(1-t)^Ne^{ithA(D)}A(D)^{N+1}adt\\
    &=\sum\limits_{k=0}^N\frac{i^kh^k}{k!}A(D)^ka+O_{N}(h^{N+1})\sum\limits_{|\alpha|\leq n+1}h^{\frac{|\alpha|}{2}}\|\partial^\alpha A(D)^{N+1}a\|_{L^\infty}.
\end{align*}
So we get
\begin{corr}\label{cor5}
Let $a,b\in S^0(T^*\mathbb{T}^n)$, then there exists a universal constant $K$ such that
$$\|\Op^w_h(a)\Op^w_h(b)-\Op^w_h(ab)\|\leq Ch\sum\limits_{|\alpha|\leq Kn}h^\frac{|\alpha|}{2}\|\partial^\alpha \sigma(D)(a\otimes b)\|_{L^\infty}$$
where $\sigma(x,\xi,y,\eta)=\langle\xi,y\rangle-\langle x,\eta\rangle$ is the standard symplectic product on $T^*\mathbb{R}^{2n}$.
\end{corr}
\begin{proof}
This follows by our previous discussion and composition formula for pseudo-differential operators \cite[Theorem 4.11]{Zw}.
\end{proof}
\subsection{Propagation of singularities}
We will study the quantitative version of propagation of singularities of Schr\"odinger equation.
First we recall an important lemma which relates Schr\"odinger equation with geodesic flow on torus.
\begin{lem}[Egorov theorem]\label{thm4}
Let $a\in S^0(\mathbb{T}^n)$, $v_h(t)=e^{-iht\Delta}$ be a unitary operator, and $\phi_t(x,\xi)=(x+2t\xi,\xi)$ be the corresponding Hamiltonian flow. Then
$$v_h(t)\Op_h^w(a)v_h(-t)=\Op_h^w(a\circ \phi_{t}).$$
\end{lem}
\begin{proof}
We recall the identity for Weyl quantization following e.g. by an explicit computation from \cite[Theorem 4.6]{Zw}
$$[-h^2\Delta, \Op^w_h(a)]=-ih\Op^w_h(\{|\xi|^2,a\}).$$
Then let $A(t)=v_h(-t)\Op_h^w(a\circ \phi_{t})v_h(t)$, we get
\begin{align*}
    \partial_t A(t)&=v_h(-t)(-\frac{i}{h}[-h^2\Delta,\Op^w_h(a\circ \phi_{t})]+\Op^w_h(2\xi\cdot\partial_x a\circ\phi_{t}))v_h(t)\\
    &=v_h(-t)(-\Op^w_h(\{|\xi|^2,a\circ \phi_{t}\})+\Op^w_h(2\xi\cdot\partial_x a\circ\phi_{t}))v_h(t)\\
    &=0.
\end{align*}
So
\begin{align*}
    A(t)=A(0)=\Op^w_h(a)
\end{align*}
or
\begin{align*}
    v_h(t)\Op_h^w(a)v_h(-t)=\Op_h^w(a\circ \phi_{t}).
\end{align*}
\end{proof}
In addition, we have
\begin{lem}\cite[Lemma 4.2]{DyJin}\label{lem2}
$$\|e^{it(-h^2\Delta-1)/h}u-u\|_{L^2}\leq \frac{|t|}{h}\|(-h^2\Delta-1)u\|_{L^2}$$
\end{lem}
\begin{proof}
It is obvious from
\begin{align*}
\partial_t e^{it(-h^2\Delta-1)/h}u=\frac{i}{h}e^{it(-h^2\Delta-1)/h}(-h^2\Delta-1)u.
\end{align*}
\end{proof}
Combine Lemma \ref{thm4} and Lemma \ref{lem2}, we have
\begin{align}\label{3.1}
    \|\Op^w_h(a\circ \phi_{t})u\|\leq \|\Op^w_h(a)u\|+\frac{|t|}{h}\|\Op^w_h(a)\|\|(-h^2\Delta-1)u\|.
\end{align}
Now we can prove a quantitative version of propagation of singularities.
\begin{prop}\label{prop8}
Let $a\in C^\infty_0(T^*\mathbb{T}^n;[0,1])$ and $b\in S^0(T^*\mathbb{T}^n;[0,1])$. If $\exists t_1,\cdots,t_M\in (0,t)$ such that $\forall p\in \supp a$, $\exists j$ such that $\phi_{t_j}(p)\in \{b= 1\}$, then $\forall u\in H^2(\mathbb{T}^n)$, we have
\begin{align*}
    \|\Op^w_h(a)u\|_{L^2}^2\leq C_{a,b,1}\|\Op^w_h(b)u\|_{L^2}^2+C_{a,b,2}\frac{|t|^2}{h^2}\|(-h^2\Delta-1)u\|_{L^2}^2+C_{a,b,3}h\|u\|_{L^2}^2
\end{align*}
where $C_{a,b,1}=C\|a\otimes b\|_{C^{Kn}_{h;M,t}}M $, $C_{a,b,2}=CM\|a\otimes b\|_{C^{Kn}_{h;M,t}}\|b\|_{C^{Kn}_h}^2$ and $C_{a,b,3}=C\|a\otimes b\|_{C^{Kn}_{h:M,t;2}}\|b\|_{C^{Kn}_h}$.

Here we use the notation:
$\|f\|_{C^k_h}=\sum\limits_{|\alpha|\leq k}h^\frac{|\alpha|}{2}\|\partial^\alpha f\|_{L^\infty}$,
$$\|a\otimes b\|_{C^k_{h;s,t}}=\sum\limits_{j\leq k}h^\frac{j}{2}\sum\limits_{l_0+l_1+\cdots+l_m=j}s^mt^{l_1+\cdots+l_m}\|a\|_{C^{l_0}}\|b\|_{C^{l_1}}\cdots\|b\|_{C^{l_m}},$$
$$\|a\otimes b\|_{C^k_{h;s,t;2}}=\sum\limits_{j\leq k}h^\frac{j}{2}\sum\limits_{l_0+l_1+\cdots+l_m=j+2}s^mt^{l_1+\cdots+l_m}\|a\|_{C^{l_0}}\|b\|_{C^{l_1}}\cdots\|b\|_{C^{l_m}}.$$
\end{prop}
\begin{proof}
Let $\chi=\sum\limits_j |b\circ \phi_{t_j}|^2\geq 1$ on $\supp a$. Let $q=\frac{|a|^2}{\chi}$, then by Lemma \ref{lem3} and Corollary~\ref{cor5}
\begin{align*}
\langle\Op^w_h(|a|^2)u,u\rangle
&\leq \sum\limits_{j}\langle \Op^w_h(\overline{\phi^*_{t_j}b})\Op^w_h(q)\Op^w_h(\phi_{t_j}^*b)u,u\rangle\\
&+Ch(\sum\limits_j\|\sigma(D)(\phi^*_{t_j}b\otimes q)\|_{C^{Kn}_h}\|b\|_{C^{Kn}_{h}}+\|\sigma(D)(q\overline{\phi^*_{t_j}b}\otimes \phi_{t_j}^*b)\|_{C^{Kn}_h})\|u\|^2\\
&\leq C\|q\|_{C^{Kn}_h}\sum\limits_{j} \|\Op^w_h(\phi_{t_j}^*b)u\|^2+Ch\|a\otimes b\|_{C^{Kn}_{h:M,t;2}}\|b\|_{C^{Kn}_h}\|u\|^2\\
&\leq C\|a\otimes b\|_{C^{Kn}_{h;M,t}}\sum\limits_{j} \|\Op^w_h(\phi_{t_j}^*b)u\|^2+Ch\|a\otimes b\|_{C^{Kn}_{h:M,t;2}}\|b\|_{C^{Kn}_h}\|u\|^2.
\end{align*}
By \eqref{3.1},
\begin{align*}
    \|\Op^w_h(\phi_{t_j}^*b)u\|-\|\Op^w_h(b)u\|&\leq \frac{|t|}{h}\|\Op^w_h(b)\|\|(-h^2\Delta-1)u\|
    \\ &\leq C\frac{|t|}{h}\|b\|_{C^{Kn}_{h}}\|(-h^2\Delta-1)u\|.
\end{align*}
So
\begin{align*}
\|\Op^w_h(a)u\|^2&=\langle \Op^w_h(|a|^2)u,u\rangle\\
&\leq C\|a\otimes b\|_{C^{Kn}_{h;M,t}}M \|\Op^w_h(b)u\|^2+Ch\|a\otimes b\|_{C^{Kn}_{h:M,t;2}}\|b\|_{C^{Kn}_h}\|u\|^2\\
&+CM\frac{t^2}{h^2}\|a\otimes b\|_{C^{Kn}_{h;M,t}}\|b\|_{C^{Kn}_h}^2\|(-h^2\Delta-1)u\|^2.
\end{align*}
\end{proof}

\section{Rational and irrational directions}
To prove Theorem \ref{thm1}, the main point is to deduce high frequency estimate by considering geodesic flow on torus. However, the dynamics on torus does not satisfy the geometric control condition. So we divide the discussion into two cases-rational and irrational as follows.

Take $\Omega_\epsilon=B(0,\epsilon)\subset\mathbb{T}^2=\mathbb{R}^2/\mathbb{Z}^2$ and $\phi_t(x,\xi)=(x+2t\xi,\xi)$ be the geodesic flow. We will always assume $\epsilon\ll 1$. We give the following definition
\begin{defi}
Two directions $\xi, \xi'\in \mathbb{R}^2_\xi\setminus\{0\}$ are equivalent iff $\xi=\lambda\xi'$ for some $\lambda\in \mathbb{R}^+$. We denote $\xi\sim\xi'$ for equivalent directions. If $\eta\sim (a,b)\in\mathbb{Z}^2\setminus\{0\}$, then it is called rational. For rational directions, define
$$L_\eta=\sqrt{a^2+b^2}$$
to be the length of the primitive geodesic in direction $\eta$, where ${\rm gcd}(a,b)=1$. If a rational direction $\eta$ satisfies
$$L_\eta^2=a^2+b^2<\frac{32}{\epsilon^2},$$
then we call $\eta$ an $\epsilon$-rational direction.
\end{defi}
\begin{prop}\label{prop9}
Let $\xi\in \mathbb{S}^1$ be a direction of length $1$. If there exists constant $C>0$ such that
\begin{align}\label{4.1}
    |\arg\xi-\arg\eta|\geq \frac{\epsilon}{CL_\eta}
\end{align}
for any $\epsilon$-rational direction $\eta$, then there exists $C'=C'(C)$ such that for any $x\in \mathbb{T}^2$, $\exists t\in [0,C'\epsilon^{-1}]$ such that $x+t\xi\in B(0,\frac{\epsilon}{3})$.
\end{prop}
\begin{proof}
Assume $\xi\sim(1,\alpha)$ with $0<\alpha<1$ and $C>12$ without loss of generality. Define $\|x\|_1:=\min\limits_{m\in\mathbb{Z}}|x-m|$, consider $\{n\alpha\mod 1:1\leq n\leq \frac{3C}{\epsilon}\}$, by Pigeonhole Principle there exist $\frac{3C}{\epsilon}\geq n'>n''\geq 1$ such that
$$\|n'\alpha-n''\alpha\|_1\leq \frac{\epsilon}{2C}.$$
Let $n=n'-n''$, then $n\in [1,\frac{3C}{\epsilon}-1]$ and there exists $m\in\mathbb{Z}$ such that
$$|n\alpha-m|\leq \frac{\epsilon}{2C}.$$
We have $0\leq m\leq n$, assume ${\rm gcd}(n,m)=1$ without loss of generality. Now let $\eta=\frac{(n,m)}{\sqrt{n^2+m^2}}$, then
\begin{align*}
    \frac{\epsilon}{2C}&\geq |n\alpha-m|\\&=|(n,m)\times (1,\alpha)|\\&=L_\eta\sqrt{1+\alpha^2}|\sin(\arg\xi-\arg\eta)|\\&\geq \frac{2}{\pi}L_\eta|\arg\xi-\arg\eta|.
\end{align*}
So $|\arg\xi-\arg\eta|\leq \frac{\pi\epsilon}{4CL_\eta}$, which means that $\eta$ is not $\epsilon$-rational by condition \eqref{4.1} (i.e. $\frac{4\sqrt{2}}{\epsilon}\leq L_\eta\leq \frac{3\sqrt{2}C}{\epsilon}$).
The intersection of the closed trajectory $\gamma=\{(tn,tm):t\in [0,1]\}\subset\mathbb{T}^2$ with the circle $\{x_1=0\}$ is given by $\{(0,\frac{k}{n}):0\leq k<n\}$. Thus each ball of radius $r>\frac{1}{2n}$ has to intersect $\gamma$. Since $n\geq\frac{L_\eta}{\sqrt{2}}\geq \frac{4}{\epsilon}$, there exists $t\in [0,L_\eta]$, $x+t\eta\in B(0,\frac{\epsilon}{7})$ for any $x\in \mathbb{T}^2$.

Now
\begin{align*}
    |(x+t\xi)-(x+t\eta)|\leq L_\eta |\xi-\eta|\leq L_\eta |\arg\xi-\arg\eta|<\frac{\epsilon}{C}<\frac{\epsilon}{12}.
\end{align*}
So $x+t\xi\in B(0,\frac{\epsilon}{3})$, i.e. $C'=3\sqrt{2}\max(C,12)$ would work.
\end{proof}
In the following section, we will prove Theorem \ref{thm1} by considering rational and irrational directions. First we note that for $\psi\in C^\infty_0(\mathbb{R};[0,1])$ such that $\psi(x)=1$ on $[-K,K]$, we have
\begin{align*}
    \|(1-\psi)(-\Delta-h^{-2})u\|_{L^2(\mathbb{T}^2)}\leq \frac{1}{K}\|(-\Delta-h^{-2})u\|_{L^2(\mathbb{T}^2)}.
\end{align*}
So we only need to consider the case when the frequency is close to $h^{-1}$. We choose a cutoff function $a\in C^\infty_0(N_\epsilon;[0,1])$ such that $a=1$ on $N_{\frac{\epsilon}{2}}$ where $N_\epsilon=    \{1-\epsilon^3<|\xi|<1+\epsilon^3\}$ and $|\partial^\alpha a|\leq C_{\alpha}\epsilon^{-3|\alpha|}$. Furthermore, we make a partition of unity $$a(\xi)=a_{\rm irr}(\xi)+\sum\limits_{\eta} a_\eta(\xi)^2$$ requiring the following conditions, where the sum is over all $\epsilon$-rational directions $\eta$.
\begin{itemize}
    \item For any $\epsilon$-rational direction $\eta$, there exists $a_\eta$ such that $a_\eta^2=a$ on $\{\xi\in N_\epsilon:|\arg\xi-\arg \eta|<\frac{\epsilon}{25L_\eta}\}$ and $a_\eta=0$ outside $\{\xi\in N_\epsilon:|\arg\xi-\arg \eta|<\frac{\epsilon}{24L_\eta}\}$. In addition, $$\|a_\eta\|_{C^k}\leq C_k\max \left(\left(\frac{\epsilon}{L_\eta}\right)^k,\epsilon^{-3k}\right)\leq C_k\epsilon^{-3k}.$$
    These $a_\eta$'s are called rational. Their number is $O(\epsilon^{-2})$.
    \item Define $a_{\rm irr}:=a-\sum\limits_{\eta}a_\eta^2$ to be the irrational part. It also satisfies
    $$\|a_{\rm irr}\|_{C^k}\leq C_k\epsilon^{-3k}.$$
\end{itemize}
We claim that any two rational $a_\eta$ and $a_{\eta'}$ have disjoint support. If $\exists\xi\in\supp a_\eta\cap\supp a_{\eta'}$, assume $L_\eta\leq L_{\eta'}$ and $\eta\sim (1,\frac{q}{p})$ and $\eta'\sim(1,\frac{q'}{p'})$ such that $0<q< p$ and $0<q'<p'$ without loss of generality, then $$\frac{1}{pp'}\leq\left|\frac{q}{p}-\frac{q'}{p'}\right|< 2|\arg \eta-\arg\eta'|\leq\frac{2\epsilon}{12L_\eta}\leq \frac{\epsilon}{6p}.$$
Therefore, $\frac{1}{pp'}<\frac{\epsilon}{6p}$, which means that $p'>\frac{6}{\epsilon}$, contradictory to that $\eta'$ is $\epsilon$-rational. See Figure \ref{fig1} for a picture of the rational cutoff functions.
\begin{figure}[h]
\centering
\includegraphics[scale=0.25]{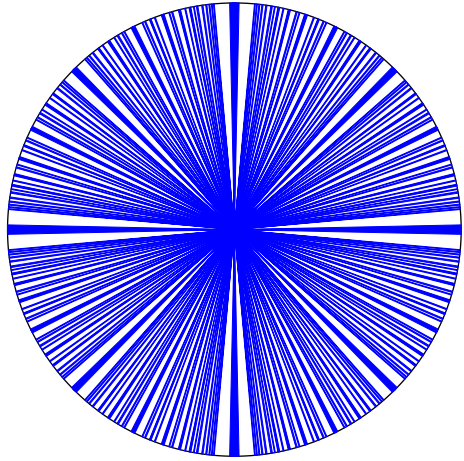}
\caption{The blue set contains the union of the supports of the rational cutoff functions $a_\eta$}\label{fig1}
\end{figure}
\section{Proof of semi-classical observability}
In this section we give the proof of Theorem \ref{thm1}. First we deal with the irrational case.
\begin{prop}
Let $\mathbb{T}^2=\mathbb{R}^2/\mathbb{Z}^2$ and $h=O(\epsilon^{8+\delta})$ for some $\delta>0$, then for $\epsilon\ll 1$, we have the following estimate
\begin{align}\label{5.1}
        \|\Op^w_h(a_{\rm irr})u\|_{L^2(\mathbb{T}^2)}^2\leq C\epsilon^{-2}\|u\|_{L^2(\Omega_\epsilon)}^2+C\epsilon^{-4}h^{-2}\|(-h^2\Delta-1)u\|_{L^2(\mathbb{T}^2)}^2+C\epsilon^{-8}h\|u\|_{L^2(\mathbb{T}^2)}^2.
\end{align}
\end{prop}
\begin{proof}
We choose a cutoff function $\chi\in C^\infty_0(\Omega_\epsilon)$ with $\|\chi\|_{C^k}\leq C_k\epsilon^{-k}$ and $\chi=1$ on $B(0,\frac{2\epsilon}{3})$. On $\supp(a_{\rm irr})$ the assumption of Proposition \ref{prop9} is satisfied for $C=25$, so there exists $C'$ such that for any $p\in \supp a_{\rm irr}$, there exists $t\in [0,C'\epsilon^{-1}]$ such that $\phi_t(p)\in T^*B(0,\frac{\epsilon}{3})$. Let $t_j=j\frac{\epsilon}{12}$, $j=1,\cdots, M=\lceil 12C'\epsilon^{-2}\rceil$, then $\phi_{t_j}(p)\in T^*B(0,\frac{2\epsilon}{3})\subset\{\chi=1\}$ for some $j\in [1,M]$. By Proposition \ref{prop8} we have
\begin{align*}
    \|\Op^w_h(a_{\rm irr})u\|_{L^2}^2\leq C_{a_{\rm irr},\chi,1}\|\Op^w_h(\chi)u\|_{L^2}^2+C_{a_{\rm irr},\chi,2}\frac{|t|^2}{h^2}\|(-h^2\Delta-1)u\|_{L^2}^2+C_{a_{\rm irr},\chi,3} h\|u\|_{L^2}^2
\end{align*}
with $M=O(\epsilon^{-2})$ and $t=O(\epsilon^{-1})$, so
\begin{align*}
C_{a_{\rm irr},\chi,1}&\leq C\epsilon^{-2}(1+h^\frac{1}{2}\epsilon^{-4})^{2K},\\
C_{a_{\rm irr},\chi,2}&\leq C\epsilon^{-2}(1+h^\frac{1}{2}\epsilon^{-4})^{2K}(1+h^\frac{1}{2}\epsilon^{-1})^{4K},\\
C_{a_{\rm irr},\chi,3}&\leq C\epsilon^{-8}(1+h^\frac{1}{2}\epsilon^{-4})^{2K}(1+h^\frac{1}{2}\epsilon^{-1})^{2K}.
\end{align*}
Therefore, let $h=O(\epsilon^{8+\delta})$, then
\begin{align*}
        \|\Op^w_h(a_{\rm irr})u\|_{L^2(\mathbb{T}^2)}^2\leq C\epsilon^{-2}\|u\|_{L^2(\Omega_\epsilon)}^2+C\epsilon^{-4}h^{-2}\|(-h^2\Delta-1)u\|_{L^2(\mathbb{T}^2)}^2+C\epsilon^{-8}h\|u\|_{L^2(\mathbb{T}^2)}^2.
\end{align*}
\end{proof}
Then we deal with the rational case
\begin{prop}
Let $\mathbb{T}^2=\mathbb{R}^2/\mathbb{Z}^2$ and $h=O(\epsilon^{8+\delta})$ for some $\delta>0$, then for $\epsilon\ll 1$ and an $\epsilon$-rational direction $\eta$, we have the following estimate
\begin{align}\label{5.2}
\begin{split}
    \|\Op^w_h(a_\eta)u\|_{L^2(\mathbb{T}^2)}^2
    &\leq  C\epsilon^{-6}\|u\|_{L^2(\Omega_\epsilon)}^2+C\epsilon^{-4}h^{-4}\|\Op^w_h(a_\eta)(-h^2\Delta-1)u\|_{L^2(\mathbb{T}^2)}^2\\
    &+Ch^{-2}\epsilon^{-10}\|(-h^2\Delta-1)u\|_{L^2(\mathbb{T}^2)}^2+C\epsilon^{-14}h\|u\|_{L^2(\mathbb{T}^2)}^2.
\end{split}
\end{align}
\end{prop}
\begin{proof}
For $\epsilon$-rational direction $\eta\sim(a,b)\in\mathbb{Z}^2\setminus\{0\}$ (${\rm gcd}(a,b)=1$),
let $L=L_\eta$ be the new period and $a_\eta$ be the corresponding $\epsilon$-rational cutoff function ($\|a_\eta\|_{C^k}\leq C_k\epsilon^{-3k}$). Cover $\mathbb{T}^2$ with a larger square with edges in direction $\eta$ and $\eta^\perp$. The square has area $a^2+b^2$ and induces a torus $\tilde{\mathbb{T}}^2=\tilde{\mathbb{T}}_{\eta}^1\times\tilde{\mathbb{T}}_{\eta^\perp}^1$. We extend the function $u$ to the larger torus periodically.
\begin{figure}[ht]
    \centering
    \includegraphics[scale=3]{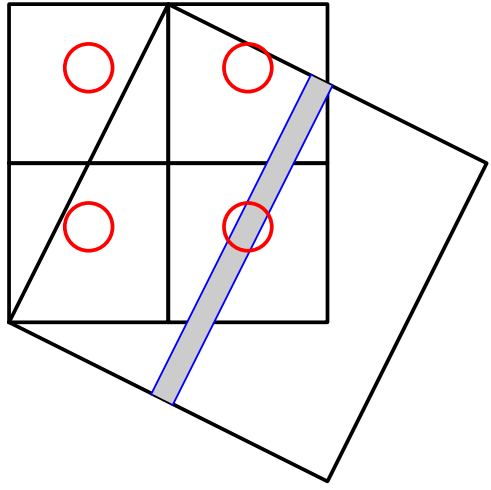}
    \caption{Cover the standard torus with a larger one. The cutoff function $b_\eta$ restricts the consideration into a small strip coming through the small ball.}\label{fig2}
\end{figure}
Let $b_\eta\in C^\infty(\tilde{\mathbb{T}}_{\eta^\perp}^1)$ such that ${\rm supp}b_\eta\subset (-\frac{\epsilon}{3},\frac{\epsilon}{3})$ and $b_\eta=1$ on $(-\frac{\epsilon}{4},\frac{\epsilon}{4})$ with $\|b_\eta\|_{C^k}\leq C_k\epsilon^{-k}$. See Figure \ref{fig2} for the covering and cutoff.

By Proposition \ref{prop1} we have
\begin{align*}
\|\Op^w_h(a_\eta)u||_{L^2(\tilde{\mathbb{T}}^2)}^2&\leq C\left(\frac{L}{\epsilon}\right)^3\|\Op^w_h(b_\eta)\Op^w_h(a_\eta)u\|_{L^2(\tilde{\mathbb{T}}^2)}^2\\
&+4L^4h^{-4}\|\Op^w_h(a_\eta)(-h^2\Delta-1)u\|_{L^2(\tilde{\mathbb{T}}^2)}^2.
\end{align*}
Moreover,
\begin{align*}
    &\|(\Op^w_h(a_\eta b_\eta)-\Op^w_h(b_\eta)\Op^w_h(a_\eta))u\|_{L^2(\tilde{\mathbb{T}}^2)}\\
    &\leq Ch\sum\limits_{|\alpha|\leq 2K}h^{\frac{|\alpha|}{2}}\|\partial^\alpha \sigma(D)(b_\eta\otimes a_\eta)\|_{L^\infty}\|u\|_{L^2(\tilde{\mathbb{T}}^2)}\\
    &\leq Ch\epsilon^{-4}(1+h^\frac{1}{2}\epsilon^{-3})^{2K}\|u\|_{L^2(\tilde{\mathbb{T}}^2)}.
\end{align*}
We then choose a cutoff function $\chi\in C^\infty_0(\Omega_\epsilon)$ with $\|\chi\|_{C^k}\leq C_k\epsilon^{-k}$ and $\chi=1$ on $B(0,\frac{2\epsilon}{3})$. Notice that for any $p\in \supp a_\eta b_\eta$, there exists $t=O(L)\leq C'\epsilon^{-1}$ such that $\phi_t(p)\in T^*B(0,\frac{9\epsilon}{24})$. Let $t_j=j\frac{\epsilon}{12}$, $j=1,\cdots, M=\lceil 12C'\epsilon^{-2}\rceil$, then $\phi_{t_j}(p)\in T^*B(0,\frac{2\epsilon}{3})\subset\{\chi=1\}$ for some $j\in [1,M]$. By Proposition \ref{prop8},
\begin{align*}
    \|\Op^w_h(a_\eta b_\eta)u\|_{L^2}^2&\leq C_{a_\eta b_\eta,\chi,1}\|\Op^w_h(\chi)u\|_{L^2}^2
    \\&+C_{a_\eta b_\eta,\chi,2}\frac{|t|^2}{h^2}\|(-h^2\Delta-1)u\|_{L^2}^2+C_{a_\eta b_\eta,\chi,3} h\|u\|_{L^2}^2
\end{align*}
with $M=O(\epsilon^{-2})$ and $t=O(\epsilon^{-1})$, so
\begin{align*}
C_{a_\eta b_\eta ,\chi,1}&\leq C\epsilon^{-2}(1+h^\frac{1}{2}\epsilon^{-4})^{2K},\\
C_{a_\eta b_\eta,\chi,2}&\leq C\epsilon^{-2}(1+h^\frac{1}{2}\epsilon^{-4})^{2K}(1+h^\frac{1}{2}\epsilon^{-1})^{4K},\\
C_{a_\eta b_\eta,\chi,3}&\leq C\epsilon^{-8}(1+h^\frac{1}{2}\epsilon^{-4})^{2K}(1+h^\frac{1}{2}\epsilon^{-1})^{2K}.
\end{align*}
So for $h=O(\epsilon^{8+\delta})$,
\begin{align*}
    \|\Op^w_h(a_\eta b_\eta)u\|_{L^2(\tilde{\mathbb{T}}^2)}^2&\leq C\epsilon^{-2}\|\Op^w_h(\chi)u\|_{L^2(\tilde{\mathbb{T}}^2)}^2\\
    &+C\epsilon^{-4}h^{-2}\|(-h^2\Delta-1)u\|_{L^2(\tilde{\mathbb{T}}^2)}^2+\epsilon^{-8}h\|u\|_{L^2(\tilde{\mathbb{T}}^2)}^2
\end{align*}
and
\begin{align*}
\|\Op^w_h(b_\eta)\Op^w_h(a_\eta )u\|_{L^2(\tilde{\mathbb{T}}^2)}^2 \leq 2\|\Op^w_h(a_\eta b_\eta)u\|_{L^2(\tilde{\mathbb{T}}^2)}^2 +\epsilon^{-8}h^2\|u\|_{L^2(\tilde{\mathbb{T}}^2)}^2.
\end{align*}
Put them together
\begin{align*}
    \|\Op^w_h(a_\eta )u||_{L^2(\tilde{\mathbb{T}}^2)}^2&\leq C\left(\frac{L}{\epsilon}\right)^3\|\Op^w_h(b_\eta)\Op^w_h(a_\eta)u\|_{L^2(\tilde{\mathbb{T}}^2)}^2\\
    &+\frac{4L^4}{h^{4}}\|\Op^w_h(a_\eta)(-h^2\Delta-1)u\|_{L^2(\tilde{\mathbb{T}}^2)}^2\\
    &\leq C\epsilon^{-5}L^3 \|u\|_{L^2(\Omega_\epsilon)}^2+C\frac{\epsilon^{-10}}{h^{2}}\|(-h^2\Delta-1)u\|_{L^2(\tilde{\mathbb{T}}^2)}^2\\
    &+\frac{C\epsilon^{-4}}{h^{4}}\|\Op^w_h(a_\eta)(-h^2\Delta-1)u\|_{L^2(\tilde{\mathbb{T}}^2)}^2+C\epsilon^{-14}h\|u\|_{L^2(\tilde{\mathbb{T}}^2)}^2.
\end{align*}
Therefore,
\begin{align*}
    \|\Op^w_h(a_\eta)u\|_{L^2(\mathbb{T}^2)}^2
    &\leq  C\epsilon^{-6}\|u\|_{L^2(\Omega_\epsilon)}^2+C\epsilon^{-4}h^{-4}\|\Op^w_h(a_\eta)(-h^2\Delta-1)u\|_{L^2(\mathbb{T}^2)}^2\\
    &+Ch^{-2}\epsilon^{-10}\|(-h^2\Delta-1)u\|_{L^2(\mathbb{T}^2)}^2+C\epsilon^{-14}h\|u\|_{L^2(\mathbb{T}^2)}^2.
\end{align*}
\end{proof}
We combine the irrational estimate \eqref{5.1} and rational estimate \eqref{5.2} to get
\begin{proof}[Proof of Theorem \ref{thm1}]
The number of $\epsilon$-rational directions is $O(\epsilon^{-2})$. For $h=O(\epsilon^{16+\delta})$ we have
\begin{align*}
    \|u\|_{L^2(\mathbb{T}^2)}^2
    &\leq \langle\Op^w_h(a)u,u\rangle+\langle\Op^w_h(1-a)u,u\rangle\\
    &\leq\|\Op^w_h(a_{\rm irr})u\|_{L^2(\mathbb{T}^2)}\|u\|_{L^2(\mathbb{T}^2)}+ \sum\limits_{\eta}\|\Op^w_h(a_\eta)u\|_{L^2(\mathbb{T}^2)}^2\\
    &+C\epsilon^{-6}\|(-h^{2}\Delta-1)u\|_{L^2(\mathbb{T}^2)}^2\\
    &\leq C(\|\Op^w_h(a_{\rm irr})u\|_{L^2(\mathbb{T}^2)}^2+ \sum\limits_{\eta}\|\Op^w_h(a_\eta)u\|_{L^2(\mathbb{T}^2)}^2+C\epsilon^{-6}\|(-h^{2}\Delta-1)u\|_{L^2(\mathbb{T}^2)}^2)\\
    &\leq C\epsilon^{-8}\|u\|_{L^2(\Omega_\epsilon)}^2+C\frac{\epsilon^{-4}}{h^4}\|(-h^2\Delta-1)u\|_{L^2(\mathbb{T}^2)}^2.
\end{align*}
This ends the proof of Theorem \ref{thm1}.
\end{proof}
\section{Classical observability}
In this section, we deduce the classical observability estimate from the semiclassical estimate on $\mathbb{T}^2=\mathbb{R}^2/\mathbb{Z}^2$. Let $\{\varphi_{\lambda,k}=e^{2\pi i(px+qy)}\}$ be eigenfunctions of $-\Delta$ with respect to eigenvalue $\lambda^2$ (i.e. $\lambda^2=4\pi^2(p^2+q^2)$) such that $\{\varphi_{\lambda,k}\}$ forms an orthonormal basis of $L^2(\mathbb{T}^2)$. Let $\Pi_N=\sum\limits_{\lambda= N}\langle u,\varphi_{\lambda,k}\rangle\varphi_{\lambda,k}$ and $\Pi_{\leq N }=\sum\limits_{\lambda\leq N}\Pi_\lambda$. Similarly, $\Pi_{> N }=\sum\limits_{\lambda> N}\Pi_\lambda$.
\subsection{High frequency estimate}
We first prove the high frequency estimate.
\begin{thm}\label{thm5}
Let $\Omega_\epsilon=B(0,\epsilon)\subset\mathbb{T}^2=\mathbb{R}^2/\mathbb{Z}^2$ and $\Pi=\Pi_{>h_0^{-1}}$ where $h_0=\epsilon^{16+\delta}$ for $\delta>0$ and $\epsilon\ll 1$, then we have
$$\|\Pi u_0\|_{L^2(\mathbb{T}^2)}^2\leq C\epsilon^{-8}\int_0^\frac{1}{2\pi}\|e^{it\Delta} \Pi u_0\|_{L^2(\Omega_\epsilon)}^2dt.$$
\end{thm}
\begin{proof}
Let $u_0=\Pi u_0$, $\chi\in C^\infty_0((0,1))$, $u=e^{it\Delta}u_0$ and $v=\chi_T(t)u$ where $\chi_T(t)=\chi(\frac{t}{T})$. We argue similarly to \cite[Proposition 3.1]{BBZ},
$$(i\partial_t-\Delta)v=\frac{i}{T}\chi_T'(t)u.$$
Take Fourier transform in $t$,
$$(-\tau-\Delta)\hat{v}=\frac{i}{T}\widehat{(\chi_T'(t)u)}.$$
For $\tau>h_0^{-2}$, apply Theorem \ref{thm1}, we have
$$\|\hat{v}\|_{L^2(\mathbb{T}^2)}\leq \frac{C\epsilon^{-2}}{ T}\|\widehat{(\chi_T'(t)u)}\|_{L^2(\mathbb{T}^2)}+C\epsilon^{-4}\|\hat{v}\|_{L^2(\Omega_\epsilon)}.$$
This is obviously true for $\tau\leq h_0^{-2}$. So
$$\|\chi_T(t)u\|_{L^2(\mathbb{T}^2\times \mathbb{R})}\leq \frac{C\epsilon^{-2}}{T}\|\chi_T'(t)u\|_{L^2(\mathbb{T}^2\times \mathbb{R})}+C\epsilon^{-4}\|\chi_T(t)u\|_{L^2(\Omega_\epsilon\times \mathbb{R})},$$
$$\|\chi\|\|u_0\|_{L^2(\mathbb{T}^2)}\leq \frac{C\epsilon^{-2}}{T}\|\chi'\|\|u_0\|_{L^2(\mathbb{T}^2)}+C\epsilon^{-4}\|\chi(t)u(tT,x)\|_{L^2(\Omega_\epsilon\times \mathbb{R})}.$$
So for appropriate $T=O(\epsilon^{-2})$, we have
$$\|u_0\|_{L^2(\mathbb{T}^2)}^2\leq C\frac{\epsilon^{-8}}{T}\|u\|_{L^2(\Omega_\epsilon\times (0,T))}^2.$$
Because $e^{it\Delta}u_0$ has period $\frac{1}{2\pi}$, we have
$$\frac{1}{T}\|u\|_{L^2(\Omega_\epsilon\times (0,T))}^2\leq C\int_0^\frac{1}{2\pi}\|e^{it\Delta} \Pi u_0\|_{L^2(\Omega_\epsilon)}^2dt.$$
This ends the proof.
\end{proof}
\subsection{Low frequency estimate} Then we estimate the low frequency part.
By \cite[Theorem 317]{HaWr}, we have $\rank \Pi_\lambda\leq e^{C\log\lambda/\log\log\lambda}$. Now we want to determine the constant in the following estimate
$$\|\Pi_\lambda u_0\|^2\leq C(\epsilon,\lambda)\|\Pi_\lambda u_0\|^2_{L^2(\Omega)}.$$
We use the following lemma
\begin{lem}[Nazarov-Tur\'{a}n lemma]\cite{Na}
Let $\mathbb{T}^1=\mathbb{R}/\mathbb{Z}$, $E\subset \mathbb{T}^1$ be a measurable subset, and $p(x)=\sum\limits_{k=1}^n c_ke^{2\pi i kx}$ be a trigonometric polynomial in $n$ characters, then exists numerical constant $C$ such that
$$\|p\|_{L^2(\mathbb{T}^1)}\leq \left(\frac{C}{|E|}\right)^{n-1}\|p\|_{L^2(E)}.$$
\end{lem}
\begin{corr} Let $\mathbb{T}^2=\mathbb{R}^2/\mathbb{Z}^2$ and $\lambda\leq h_0^{-1}=\epsilon^{-17}$, there exists constant $C$ such that
$$\|\Pi_\lambda u_0\|^2_{L^2(\mathbb{T}^2)}\leq e^{e^{C\log{\epsilon^{-1}}/\log\log{\epsilon^{-1}}}}\|\Pi_\lambda u_0\|^2_{L^2(\Omega_\epsilon)}.$$

\end{corr}
\begin{proof}
Let $u_0=\Pi_\lambda u_0=\sum\limits_{k=1}^n c_k\varphi_{\lambda,k}$ where $n\leq \rank\Pi_\lambda\leq e^{C\log\lambda/\log\log\lambda}$.
We first fix $y$ and apply Nazarov-Tur\'{a}n lemma to get
\begin{align*}
    \|u_0\|^2_{L^2(\mathbb{T}^1_x\times \{y\})}\leq (C\epsilon^{-1})^{2n}\|u_0\|^2_{L^2([-\frac{\epsilon}{2},\frac{\epsilon}{2}]\times \{y\})}.
\end{align*}
By integrating it on $y$, we get
\begin{align*}
    \|u_0\|^2_{L^2(\mathbb{T}^2)}\leq (C\epsilon^{-1})^{2n}\|u_0\|^2_{L^2([-\frac{\epsilon}{2},\frac{\epsilon}{2}]\times \mathbb{T}^1_y)}.
\end{align*}
Similarly, we have
\begin{align*}
    \|u_0\|^2_{L^2([-\frac{\epsilon}{2},\frac{\epsilon}{2}]\times \mathbb{T}^1_y)}\leq (C\epsilon^{-1})^{2n}\|u_0\|^2_{L^2([-\frac{\epsilon}{2},\frac{\epsilon}{2}]\times [-\frac{\epsilon}{2},\frac{\epsilon}{2}])}.
\end{align*}
So in conclusion, we get
\begin{align*}
    \|u_0\|^2_{L^2(\mathbb{T}^2)}&\leq (C\epsilon^{-1})^{4n}\|u_0\|^2_{L^2([-\frac{\epsilon}{2},\frac{\epsilon}{2}]^2)}\\
    &\leq (C\epsilon^{-1})^{4n}\|u_0\|^2_{L^2(\Omega_\epsilon)}\\
    &\leq e^{e^{C\log\epsilon^{-17}/\log\log\epsilon^{-17}+\log\log (C\epsilon^{-1})}}\|u_0\|^2_{L^2(\Omega_\epsilon)}\\
    &\leq e^{e^{C\log\epsilon^{-1}/\log\log\epsilon^{-1}}}\|u_0\|^2_{L^2(\Omega_\epsilon)}.
\end{align*}
\end{proof}
Now for any eigenvalues $\lambda\neq \mu\in \Spec(\sqrt{-\Delta})$,
$$\int_0^\frac{1}{2\pi} \langle e^{it\Delta}\Pi_\lambda u_0, e^{it\Delta}\Pi_\mu u_0\rangle_{L^2(\Omega_\epsilon)} dt=0.$$
So we get
\begin{align*}
    \|\Pi_{\leq h_0^{-1}} u_0\|^2_{L^2(\mathbb{T}^2)}\leq e^{e^{C\log{\epsilon^{-1}}/\log\log{\epsilon^{-1}}}}\int_0^\frac{1}{2\pi}\|e^{it\Delta}\Pi_{\leq h_0^{-1}} u_0\|^2_{L^2(\Omega_\epsilon)}dt.
\end{align*}
Combine this with Theorem \ref{thm5} we get
\begin{align*}
    \|u_0\|^2_{L^2(\mathbb{T}^2)}\leq e^{e^{C\log{\epsilon^{-1}}/\log\log{\epsilon^{-1}}}}\int_0^\frac{1}{2\pi}\|e^{it\Delta}u_0\|_{L^2(\Omega_\epsilon)}^2dt.
\end{align*}
This ends the proof of Theorem \ref{thm2}.

\end{document}